\documentclass[11pt,a4paper,reqno]{amsart}
\usepackage{amssymb,amsmath,amsthm}
\usepackage[utf8]{inputenc}
\usepackage[T1]{fontenc}
\usepackage{enumerate}
\usepackage[all]{xy}
\usepackage{fullpage}
\usepackage{comment}
\usepackage{array}
\usepackage{longtable}
\usepackage{stmaryrd}
\usepackage{mathrsfs} 
\usepackage{hyperref}
\usepackage{xcolor}
\usepackage{soul}

\def\N{\mathbb{N}}

\def\F{\mathbb{F}}

\def\deg{\mathrm{deg}}

\def\int{\mathrm{Int}}

\def\deg{\mathrm{deg}}
\def\cal{\mathcal}

{\theoremstyle{plain}
\newtheorem{theo}{Theorem}[section]
\newtheorem{lem}[theo]{Lemma}
\newtheorem{coro}[theo]{Corollary}

\newtheorem{conj}[theo]{Conjecture}
}

{\theoremstyle{remark}

\newtheorem{defi}[theo]{Definition}

\newtheorem{question}[theo]{Question}

}

\title{On the stopping time of the Collatz map in $\F_2[x]$}

\author{Gil Alon}
\email{gilal@openu.ac.il}
\address{Department of Mathematics and Computer Science, The Open University of Israel}

\author{Angelot Behajaina}
\email{abeha@campus.technion.ac.il}
\address{Department of Mathematics, Technion - Israel Institute of Technology, Haifa, Israel}
\address{Department of Mathematics and Computer Science, The Open University of Israel}

\author{Elad Paran}
\email{paran@openu.ac.il}
\address{Department of Mathematics and Computer Science, The Open University of Israel}

\begin{document}
\maketitle

\noindent

\begin{abstract}
    We study the stopping time of the Collatz map for a polynomial $f \in \F_2[x]$, and bound it by $O(\deg (f)^{1.5})$, improving upon the quadratic bound proven by Hicks, Mullen, Yucas and Zavislak. We also prove the existence arithmetic sequences of unbounded length in the stopping times of certain sequences of polynomials, a phenomenon observed in the classical Collatz map.
    
\end{abstract}

\maketitle
\section{Introduction}

The classical Collatz map is defined, for any natural number $n$, by ${\cal{C}}(n) = n/2$ if $n$ is even, and ${\cal{C}}(n) = 3n+1$ if $n$ is odd. The Collatz conjecture states that for any $n \geq 1$ there exists $k \geq 1$ for which ${\cal{C}}^k(n) = 1$. The minimal such $k$, if exists, is called the {\bf stopping time}\footnote{In some texts this is known as the {\bf total} stopping time.} of $n$ with respect to ${\cal{C}}$. If no such $k$ exists, one says that the stopping time of $n$ is infinite. Thus the Collatz conjecture states that the stopping time of any $n \in \N$ is finite. While some interesting partial results have been obtained (most notably in the work \cite{Tao2022} of Tao), the conjecture remains wide open and a proof is considered far from reach. For a survey on the history of this problem, see \cite{Lagarias2010}.

As in the case of many arithmetic questions, it is natural to look for analogues in polynomial rings. Such an analogue was studied in 2008 by Hicks, Mullen, Yucas and Zavislak in \cite{Hicks}, who replaced the integer prime $2$ with the prime $x$ of the ring $\mathbb{F}_2[x]$, thus defining the {\bf polynomial Collatz map} $T \colon \mathbb{F}_2[x]\rightarrow \mathbb{F}_2[x]$ by
\begin{equation}\label{eq:defcollmap}
T(f)=\begin{cases}
(1+x)f+1 & f\,\text{is odd},\\
f/x & f\,\text{is even}.
\end{cases}
\end{equation}
Here, a polynomial $f\in \mathbb{F}_2[x]$ is called \emph{odd} if $f(0)=1$, and
\emph{even} if $f(0)=0$. That is, $f$ is even if and only if $f$ is divisible by $x$.

For any non-zero $f \in \mathbb{F}_2[x]$, we define the stopping time of $f$ (with respect to the above map) as the minimal $k \geq 1$ for which $T^k(f) = 1$. The main result of Hicks et al. is that for any non-zero polynomial $f\in \mathbb{F}_2[x]$, the stopping time of $f$ is finite. Moreover, they prove that this stopping time is at most $\deg(f)^2+2\deg(f)$. 

It is interesting to compare this result with conjectures concerning the stopping time of the arithmetic Collatz map ${\cal{C}}$, where the analogue of the degree of a polynomial is the number of digits of a number $n$, i.e. $\log(n)$. Based on heuristic probabilistic arguments, both Crandall \cite{Crandall} and Shanks (see \cite[p.~13]{Lagarias85}) conjectured that the {\bf average} stopping time of a natural number $n$ grows linearly in $\log(n)$. We do not know if a similar linear bound holds for the average, or maximal, stopping time of polynomials in $\F_2[x]$ of a given degree (see Question \ref{worse} and Conjecture \ref{average} below). However, we are able to significantly improve upon the bound of Hicks et al. The main result of this work is the following:

\begin{theo}\label{main}
There exists a constant $c$ such that for all $f \in \F_2[x]$, the stopping time of $f$ with respect to the Collatz map is at most $c\cdot \deg(f)^{1.5}$. 
\end{theo}

Our second result is related to a phenomenon observed in the stopping times of certain numbers with respect to the classical map ${\cal{C}}$. It was observed in \cite{MO} that the the sequence of stopping times of integers of the form $2^n+1$ contains arithmetic progressions of increasing length as $n \to \infty$. The same phenomenon is observed in other similar sequences, which lead to the following conjecture in \cite{MO}: 

\begin{conj}\label{conj_main} For fixed integers $a,b \geq 0$, the sequence of stopping times of the Collatz map of numbers of the form $(2^a3^b)^n+1$ contains arithmetic sequences of unbounded length, and with common difference $a-b$. \end{conj}

Here, the term 'contains arithmetic progressions of unbounded length' is given precisely by

\begin{defi}
Let $(a_{n})_{n \geq 0}$ and let $r \geq 0$. We say that $(a_{n})_{n \geq 0}$ has \emph{arithmetic sequences of unbounded length and with common difference $r$} if: for all $m \geq 1$, there exists $\ell \geq 0$, such that 
$$a_{\ell+1}=a_{\ell}+r, a_{\ell+2}=a_{\ell}+2r, \dots, a_{\ell+m-2}=a_{\ell}+(m-2)r, a_{\ell+m-1}=a_{\ell}+(m-1)r.$$
\end{defi}

While Conjecture \ref{conj_main} is concerned with a rather specific family of integers, there does not seem to be an easy way to analyse the stopping time of these integers in general and prove the conjecture in the arithmetic case. Here, we consider the analogous conjecture over $\F_2[x]$, and analyse the stopping time of polynomials of the form $\big(x^a(1+x)^b\big)^n+1$. We extract a formula for the stopping times of these polynomials, see Corollary \ref{coro:stoppingtime} below. As a consequence of this formula, we obtain the following polynomial analogue of Conjecture \ref{conj_main}:

\begin{theo}\label{main2} Let $a,b$ be non-negative integers, with $a > 0$ or $b > 0$. The sequence of stopping times of polynomials of the form $\big(x^a(1+x)^b\big)^n+1$ contains arithmetic sequences of unbounded length, and with common difference $a-b$. \end{theo}

Our first step in the proof of the above results is a reduction of the map $T$ on $\mathbb{F}_2[x]$ to an equivalent map which is simpler to handle. This is based on the following observations: 

Consider a polynomial $f \in \mathbb{F}_2[x]$, 
$$f=a_0+a_1x+\dotsc +a_nx^n \hspace{10pt} (a_n\neq 0)$$and let 
$$g=(1+x)f=a_0+(a_0+a_1)x+\dotsc+(a_{n-1}+a_n)x^n+a_nx^{n+1}.$$
Let us consider the polynomials $F$ and $G$ obtained by reversing the order of the coefficients of $f$ and $g$, respectively:
$$F=a_n+a_{n-1}x+\dotsc +a_0x^n,$$
$$G=a_n+(a_{n-1}+a_n)x+\dotsc+(a_0+a_1)x^n+a_0x^{n+1}.$$
Then $G=(1+x)F$. That is, the operation of multiplying by $1+x$ commutes with the operation of reversing the order of the coefficients. By this observation, given a polynomial $f$, instead of looking at its iterates under the mapping $T$, $$f_i=T^i(f)\hspace{10pt}(i=0,1,2,\dotsc)$$
one can consider the polynomials $F_i$ obtained from the $f_i$ by reversing the order of their coefficients. As we analyse in \S\ref{sec:stopping_time}, the polynomials $F_i$ undergo the following process: For each $i\geq 0$, either $F_{i+1}=F_i$ or $F_{i+1}=R(F_i)$, where $R(F)$ is the polynomial obtained from $(1+x)F$ by removing its leading term. Therefore, the stopping time analysis of the Collatz process is reduced to the analysis of the iterates of the map $R$. This map has the advantage that for each $F\in \mathbb{F}_2[x]$ and $i\geq 0$, $R^i(F)$ is equal to $(1+x)^i F$ with some of its leading terms removed (see Lemma \ref{lem:restriction} below). 

Taking advantage of this reduction, we then apply certain algebraic and combinatorical considerations in order to prove our main results. The proof of Theorem \ref{main} is given in \S\ref{sec:stopping_time}, and the proof of Theorem \ref{main2} is given in \S\ref{sec:arithmetic_progression}. Finally, in \S\ref{sec:open} we highlight several open questions for further study.

\section*{Acknowledgements}
The authors thank Alon Rosenperl for his experimental findings concerning Question \ref{average}. The second author is grateful for the support of a Technion fellowship, of an Open University of Israel post-doctoral fellowship, of the Israel Science Foundation (grant no. 353/21), and of TU Dresden during a research visit in December 2023 and January 2024.

\section{Bounding the stopping time of the polynomial Collatz map}\label{sec:stopping_time}
Let $T \colon \mathbb{F}_2[x] \rightarrow \mathbb{F}_2[x]$ denote the polynomial Collatz map defined in the introduction. Given a nonzero polynomial $f\in \mathbb{F}_2[x]$, we are interested in the stopping time of the iterative process defined by $T$ on $f$, i.e. $$t_{{\rm min}}(f)=\min\{k\geq 0 \mid T^{ k}(f)=1\}.$$ 
Consider the coefficient reversing map: 
\begin{align}
\hat{} \colon \mathbb{F}_2[x]&\rightarrow \mathbb{F}_2[x] \label{eq:rever}\\
 f& \mapsto \hat{f}=x^{\deg(f)}f(1/x). \nonumber
\end{align}
(we use the convention that $\deg (0)=-\infty$, so that $\hat{0}=0$ holds). This map has the following properties, all of which are straightforward to verify:
\begin{enumerate}
\item If $f(x)=\sum_{i=0}^{m}a_{i}x^{i} \in \mathbb{F}_2[x]$ with $a_{m}\neq 0$, then $\hat{f}(x)=\sum_{i=0}^{m}a_{i}x^{m-i}$.
\item For any $f,g\in \mathbb{F}_2[x]$, we have $\widehat{fg}=\hat{f}\hat{g}$.
\item For any $f \in \mathbb{F}_2[x]$ and any integer $k \geq 0$, we have $\widehat{x^{k}f}=\hat{f}$.
\item For any $f\in \mathbb{F}_2[x]\setminus \{0\}$, we have $\hat{\hat{f}}=f/x^{r}$, where $r$ is the maximal power of $x$ dividing $f$. In particular, if $f$ is odd then $\hat{\hat{f}}=f$.
\end{enumerate}
Given $f\in \mathbb{F}_2[x]\backslash\{0\}$, in order to estimate $t_{{\rm min}}(f)$,
we consider the sequence of iterates $f,T(f),T^{ 2}(f),\ldots$ and their
transforms $\widehat{f},\widehat{T(f)},\widehat{T^{ 2}f},\ldots$ Note
that whenever $T^{ n}(f)=\left(T^{n-1}(f)\right)/x$, we have $\widehat{T^{ n}(f)}=\widehat{T^{n-1}(f)}$,
i.e. the sequence of transforms ``stutters''. 

Let us define the following auxiliary maps on $\mathbb{F}_2[x]$.

\begin{defi}
For $f \in \mathbb{F}_2[x]$, let:
\begin{itemize}
\item $T_{1}(f)=(1+x)f+1$;
\item $T_{2}(f)=f/x^{r}$, where $r$ is the maximal power of $x$ dividing
$f$ (and $T_2(0)=0)$;
\item $T_{3}(f)=(T_{2}\circ T_{1})(f)$.
\end{itemize}
\end{defi}
Clearly, the iterates $\left (T_{3}^{ n}(f) \right)_{n \geq 0}$ form a subsequence of the iterates $\left (T^{ n}(f) \right )_{n\geq 0}$.

For any map $u \colon \mathbb{F}_2[x]\rightarrow \mathbb{F}_2[x]$ and $f\in \mathbb{F}_2[x]$, let us define $$t_{\min}(f,u)=\min\{k\geq 0 \mid u^{ r}(f)=1\}$$
(if no such $k$ exists then $t_{{\rm min}}(f,u)=\infty).$ In particular, when $u=T$, we retrieve $t_{\min}$ defined at the beginning of the section.
\begin{lem}
Suppose that $f$ is odd. Then $t_{{\rm min}}(f)=2t_{{\rm min}}(f,T_{3})+\deg(f)$.
\end{lem}

\begin{proof}
Since $t_{{\rm min}}(f)$ is finite (by \cite{Hicks}), let us prove the claim by induction on it. The basis of the induction is the case $f=1$ for which the claim clearly holds. In the inductive step, let $T_{1}(f)=x^{r}g$
with $x\nmid g$. Then $T_{3}(f)=T^{r+1}(f)=g$, and by the induction
hypothesis, 
\begin{align*}
t_{\min}(f) & =r+1+t_{{\rm min}}(g)=r+1+2t_{{\rm min}}(g,T_{3})+\deg (g)\\
 & =r+1+2(t_{{\rm min}}(f,T_{3})-1)+(\deg(T_{1}(f))-r)\\
 & =r+1+2t_{{\rm min}}(f,T_{3})-2+(\deg(f)+1-r)\\
 & =2t_{{\rm min}}(f,T_{3})+\deg (f). \qedhere
\end{align*}
\end{proof}
Let us now define the following maps from on $\mathbb{F}_2[x]$.

\begin{defi}\label{defi:s1s2}
For $f \in \mathbb{F}_2[x]$, we let:
\begin{itemize}
\item $S_{1}(f)=(x+1)f$;
\item $S_{2}(f)=f+x^{\deg(f)}$ (Since we are over $\mathbb{F}_{2}$, this
map nullifies the leading term of $f$);
\item $S_{3}(f)=(S_{2}\circ S_{1})(f)$.
\end{itemize}
\end{defi}
\begin{lem}
Suppose that $f$ is odd. Then we have:
\begin{enumerate}
\item $\widehat{T_{3}(f)}=S_{3}\left(\hat{f}\right)$,
\item $t_{{\rm min}}(f,T_{3})=t_{{\rm min}}(\hat{f},S_{3})$.
\end{enumerate}
\end{lem}

\begin{proof}
~
\begin{enumerate}
\item We first note that for any odd $g$, we have $\widehat{g+1}=S_{2}(\widehat{g}).$
Using this and the properties of the $\widehat{\,\,\,}$~transform,
we have 
\begin{align*}
\widehat{T_{3}(f)} & =\widehat{T_{2}(T_{1}(f))}=\widehat{T_{1}(f)}=\widehat{(1+x)f+1}=S_{2}\left(\widehat{(1+x)f}\right)\\
 & =S_{2}\left(\widehat{x+1}\cdot\widehat{f}\right)=S_{2}\left((1+x)\widehat{f}\right)=S_{2}\left(S_{1}\left(\hat{f}\right)\right)=S_{3}\left(\hat{f}\right).
\end{align*}
\item This follows from (1), since iterates of $T_{3}$ produce odd elements,
and for such elements we have $\hat{f}=1\Leftrightarrow f=1$. \qedhere
\end{enumerate}
\end{proof}
The last lemma shows that we can replace the Collatz process
in $\mathbb{F}_2[x]$ with the simpler process of multiplying by $(x+1)$ and removing
the leading term. 

We also note that by the last two lemmas, for any odd $f$, we have
\begin{equation} \label{eq:t_min}
t_{{\rm min}}(f)=2t_{{\rm min}}(\hat{f},S_{3})+\deg(f). 
\end{equation}

Let us introduce the following notation: 

For a polynomial $f(x)=\sum a_{i}x^{i} \in \mathbb{F}_2[x]$, we let $f|_{\leq n}=\sum_{i\leq n}a_{i}x^{i}$.

\begin{lem}
\label{lem:iterations_formula}Let $f(x)=x^{n}+g$, where $\deg(g)< n$. Then for any $0\leq i \leq n-\deg(g)$, we have
$ S_3^{ i} (f)= (x+1)^ig+x^n=((x+1)^if)|_{\leq n}$.
\end{lem}

\begin{proof}
Let $r=n-\deg(g)$. Let us prove the first equality $ S_3^{ i} (f)= (x+1)^ig+x^n$, by induction on $i$. For $i=0$ it is clear. Let $i<n - \deg(g)$, and assume that $ S_3^{ i} (f)= (x+1)^ig+x^n$. We note that $\deg ( (x+1)^i g)<n$, hence $\deg (S_3^{ i}(f))=n$. Applying $S_1$, we get
\[
S_{1}(S_{3}^{ i}(f))=(x+1)S_3^{ i}(f),
\]
which is a polynomial of degree $n+1$, so
\begin{align*}
S_3^{i+1}(f) & =S_2 (S_1(S_{3}^{ i}(f)))\\
 & =S_1(S_{3}^{ i}(f))+x^{n+1}\\
 & = (x+1) ((x+1)^ig + x^n) + x^{n+1}\\
 & = (x+1)^{i+1}g +x^n. 
\end{align*}
This completes the induction. 
Finally, for any $0\leq i \leq n-\deg(g)$, we have $\deg((x+1)^ig) \leq n$ and $x^n\equiv x^n (x+1)^i \mod x^n$, hence 
\begin{align*} 
S_3^{ i}(f)&= (x+1)^ig +x^n = ((x+1)^ig + x^n)|_{\leq n} \\
&=((x+1)^ig + x^n(x+1)^i)|_{\leq n}\\
&=((x+1)^i (g+x^n))|_{\leq n} = ((x+1)^if)|_{\leq n}. \qedhere
\end{align*}
\end{proof}

\begin{lem}
\label{lem:n-r}Let $f(x)=x^{n}+g$, where $g$ is odd, $\deg(g)=n-r$ and $r>0$. Then we have $t_{{\rm min}}(f,S_{3})=r-1+t_{{\rm min}}\left((x+1)^{r-1}g,S_{3}\right) $. 
\end{lem}

\begin{proof}
By Lemma \ref{lem:iterations_formula},

\[
S_{3}^{ r}(f)=x^{n}+(x+1)^{r}g.
\]
We note that $x^{n}$ is cancelled in the last expression, so 
\begin{align*}
S_{3}^{ r}(f) & =S_{2}\left((x+1)^{r}g\right)=S_{2}\left(S_{1}\left((x+1)^{r-1}g\right)\right)=S_{3}\left((x+1)^{r-1}g\right).
\end{align*}
Again by Lemma \ref{lem:iterations_formula}, we have $S_{3}^{ i}(f)\neq 1$, for all $1\leq i\leq r-1$. Hence we get
\begin{equation*}
t_{{\rm min}}(f,S_3)=r+t_{{\rm min}}(S_3((x+1)^{r-1}g),S_3)=r-1+t_{{\rm min}}\left((x+1)^{r-1}g,S_{3}\right). \qedhere
\end{equation*}
\end{proof}

\begin{lem}\label{lem:restriction}
Let $f,g$ be odd polynomials.
\begin{enumerate}
\item We have $\deg(S_{3}(f))\leq\deg(f)$.
\item For all $k\geq0$, we have $S_{3}^{ k}(f)=\left((x+1)^{k}f\right)|_{\leq m}$
where $m=\deg(S_{3}^{ k}(f))$.
\item If $g=f|_{\leq n}$ for some $n$, then $t_{\min}(g,S_{3})\leq t_{{\rm min}}(f,S_{3})$.
\end{enumerate}
\end{lem}

\begin{proof}
~
\begin{enumerate}
\item This is clear, since $\deg(S_{3}(f))=\deg( S_{2}(S_{1}(f)))<\deg(S_{1}(f))=\deg (f)+1$.
\item The claim is clear for $k=0$. If it is true for $k$, then let $h_{1}=S_{3}^{ k}(f)$ and
$h_{2}=S_{3}^{k+1}(f)=S_{3}(h_{1})$. Let $d_{1}=\deg(h_{1})$ and $d_{2}=\deg(h_{2})$. Then, by (1), we have $d_{2}\leq d_{1}$. Since $h_{1}=\left((x+1)^{k}f\right)|_{\leq d_{1}}$ (by Lemma \ref{lem:iterations_formula}),
we have $h_{1}\equiv(x+1)^{k}f\mod x^{d_{1}+1}$, so $(x+1)h_{1}\equiv(x+1)^{k+1}f\mod x^{d_{1}+1}$,
hence $\left((x+1)^{k+1}f\right)|_{\leq d_{1}}=\left((x+1)h_{1}\right)|_{\leq d_{1}}$
. We conclude that $(x+1)^{k+1}f|_{\leq d_{2}}=\left((x+1)h_{1}\right)|_{\leq d_{2}}=S_{3}(h_{1})=h_{2}$.
\item Let $f_{k}=S_{3}^{ k}(f)$ and $g_{k}=S_{3}^{ k}(g)$ for
$k=0,1,2,\ldots$ We will see by induction that $\deg( g_{k})\leq\deg(f_{k})$
and $f_{k}\equiv g_{k}\mod x^{1+\deg(g_{k})}$ for all $k$. Indeed,
for $k=0$ it is true by our assumptions. If it is true for $k$,
then $(x+1)f_{k}\equiv(x+1)g_{k}\mod x^{1+\deg( g_{k})}$, hence
\begin{align*}
f_{k+1} & =(x+1)f_{k}+x^{1+\deg( f_{k})}\equiv(x+1)g_{k}\\
 & \equiv(x+1)g_{k}+x^{1+\deg (g_{k})}=g_{k+1}\mod x^{1+\deg(g_{k})},
\end{align*}
and since $\deg(g_{k})\geq\deg(g_{k+1})$, we have
\[
f_{k+1}\equiv g_{k+1}\mod x^{1+\deg (g_{k+1})},
\]
from which we also get that $\deg(g_{k+1})\leq\deg( f_{k+1})$. Our induction
is complete, and the claim follows. \qedhere
\end{enumerate}
\end{proof}
\begin{lem}
Let $f=(x+1)^{r}g$ where $g$ is odd. Then we have
\[
t_{{\rm min}}(f,S_{3})\leq2\deg(f)-r+t_{\min}(g,S_{3}).
\]
\label{lem:estimate}
\end{lem}

\begin{proof}
Let $s$ be an integer such that $2^{s-1}\leq\deg(f)<2^{s}$. Let $k=2^{s}-r$
and let $d=\deg\left(S_{3}^{ k}(f)\right)$. By Lemma \ref{lem:restriction},
$d\leq\deg(f)<2^{s}$ and

\begin{align*}
S_{3}^{ k}(f) & =(x+1)^{k}f|_{\leq d}=(x+1)^{2^{s}}g|_{\leq d}\\
 & =\left(x^{2^{s}}+1\right)g|_{\leq d}=x^{2^{s}}g+g|_{\leq d}\\
 & =g|_{\leq d}.
\end{align*}
Again by Lemma \ref{lem:restriction}, 
\begin{align*}
t_{{\rm min}}(f,S_{3}) & \leq k+t_{{\rm min}}(g|_{\leq d},S_{3})\leq k+t_{{\rm min}}(g,S_{3})\\
 & =2^{s}-r+t_{{\rm min}}(g,S_{3})\leq2\deg(f)-r+t_{{\rm min}}(g,S_{3}). \qedhere
\end{align*}
\end{proof}
\begin{theo}
For any odd $f\in \mathbb{F}_2[x]$ we have $t_{\min}(f,S_{3})\leq\sqrt{2}\left(\deg (f)\right)^{1.5}$.
\end{theo}

\begin{proof}
Let $n=\deg(f)$. We proceed by induction on $n$. If $n=0$
then $f=1$ and the stopping time is $0$, so the claim holds.

In the inductive step, let $f(x)=x^{n}+g$, where $\deg(g)=n-r$ and
$r>0$. 

By Lemma \ref{lem:n-r}, we have $t_{{\rm min}}(f,S_{3})=r-1+t_{{\rm min}}\left((x+1)^{r-1}g,S_{3}\right)$.
We distinguish between two cases:
\begin{itemize}
\item Case 1: $r\leq\sqrt{2n}$. We use the inductive hypothesis: $\deg\left((x+1)^{r-1}g\right)=n-1$,
so 
\[
t_{{\rm min}}(f,S_{3})\leq r-1+(n-1)\sqrt{2(n-1)}<\sqrt{2n}+(n-1)\sqrt{2n}=n\sqrt{2n}.
\]
\item Case 2:  $r>2\sqrt{n}$. By Lemma \ref{lem:estimate}, 
\begin{align*}
t_{{\rm min}}(f,S_{3}) & =r-1+t_{{\rm min}}\left((x+1)^{r-1}g,S_{3}\right)\\
 & \leq r-1+2(n-1)-(r-1)+t_{{\rm min}}(g,S_{3})\\
 & <2n+t_{{\rm min}}(g,S_{3}).
\end{align*}
By the inductive assumption, we get
\begin{align*}
t_{{\rm min}}(f,S_{3}) & <2n+(n-r)\sqrt{2(n-r)}\\
 & \leq2n+(n-r)\sqrt{2n}\\
 & <2n+(n-\sqrt{2n})\sqrt{2n}=n\sqrt{2n}.
\end{align*}
Hence, our induction is complete. \qedhere
\end{itemize}
\end{proof}

Combining this result with (\ref{eq:t_min}), we immediately get

\begin{coro}
For any odd $f \in \mathbb{F}_2[x]$, we have $t_{{\rm min}}(f)\leq(2\deg(f))^{1.5}+\deg(f)$.
\end{coro}

Finally, if $f$ is an even polynomial of degree $n$, letting $f=x^{r}g$ with $g$ odd,
we have $T^{r}(f)=g$, so
\begin{align*}
t_{{\rm min}}(f) & =r+t_{{\rm min}}(g)\leq r+(2(n-r))^{1.5}+n-r\\
 & =(2(n-r))^{1.5}+n<(2n)^{1.5}+n.
\end{align*}

Thus:
\begin{theo}
For any nonzero $f\in \mathbb{F}_2[x]$, we have $t_{{\rm min}}(f)\leq(2\deg(f))^{1.5}+\deg(f)$.
\end{theo}

This yields Theorem \ref{main} of the introduction.

\section{Arithmetic progressions in stopping times of the Collatz map}\label{sec:arithmetic_progression}
In this section, we establish Theorem \ref{main2}, which is a positive characteristic analogue of Conjecture \ref{conj_main}, as mentioned in the introduction.

Consider the polynomials in Theorem \ref{main2}: $$f_{a,b,n}=(x^a(x+1)^b)^n+1,$$ where $a,b$ are non-negative integers such that $(a,b)\neq(0,0)$ and $n\geq 1$. Recall that Theorem \ref{main2} asserts that, for all $a,b$, there exist arithmetic sequences of unbounded length in $(t_{{\rm min}}(f_{a,b,n}))_{n \geq 1}$. We will prove it by finding these stopping times explicitly.

In the following, we consider the mapping $S_3$ (Definition \ref{defi:s1s2}) and the transform $\:{\widehat{}}\:$ (\eqref{eq:rever}) from Section \ref{sec:stopping_time}.

\begin{lem} \label{lem:fab_reduce}
We have
$$ t_{{\rm min}}(f_{a,b,n}) = 2t_{\rm min}((x+1)^{na+nb-1},S_3)+3na+nb-2.$$
\end{lem}
\begin{proof}  
By (\ref{eq:t_min}) we have
\begin{equation} \label{eq:t_fabn}
t_{{\rm min}}(f_{a,b,n})=2 t_{{\rm min}}(\widehat{f_{a,b,n}},S_3)+n(a+b). 
\end{equation}
Let us find the transform $\widehat{f_{a,b,n}}$: By definition,
$$ \widehat{f_{a,b,n}}=x^{\deg ( f_{a,b,n})} f_{a,b,n}(1/x)= (x+1)^{nb}+x^{na+nb}.$$
By Lemma \ref{lem:iterations_formula},
$$ S_3^{na}(\widehat{f_{a,b,n}})=(x+1)^{na+nb}+x^{na+nb}=S_3((x+1)^{na+nb-1}).$$
Hence 
\begin{align*}
t_{\rm min}(\widehat{f_{a,b,n}},S_3)&=na+t_{\rm min}(S_3((x+1)^{na+nb-1}),S_3) \\
&= na-1+t_{\rm min}((x+1)^{na+nb-1},S_3). 
\end{align*}
Together with (\ref{eq:t_fabn}), we get the desired formula.
\end{proof}

Therefore, the calculation of $(t_{{\rm min}}(f_{a,b,n}))_{n\geq 1}$ is reduced to the calculation of \\ $(t_{{\rm min}}((x+1)^n,S_3))_{n\geq 1}$, which will be carried out in the following lemmas.

\begin{lem} \label{lem:expansion}
Let $n\geq 1$, and assume that $n$ has the following base-2 expansion:
$$n=\sum_{i=1}^{m}2^{t_i},\hspace{20pt}t_1<t_2<\dotsc<t_m.$$ Then the three leading terms of the polynomial $(x+1)^n\in \mathbb{F}_2[x]$ are
\begin{equation} \label{eq:binary}  
 (x+1)^n=x^n+x^{n-2^{t_1}}+x^{n-2^{t_2}}+\dotsc
 \end{equation}
\end{lem}
\begin{proof}
We have 
\begin{align*}
(x+1)^n&=(x+1)^{\sum_{i=1}^m2^{t_i}}=\prod_{i=1}^m\left(x+1\right)^{2^{t_i}}=\prod_{i=1}^m\left(x^{2^{t_i}}+1\right)\\
&=\sum_{(\varepsilon_1,...,\varepsilon_m)\in \{0,1\}^m} x^{\varepsilon_12^{t_1}+\varepsilon_22^{t_2}+\dotsc+
\varepsilon_m2^{t_m}}.
\end{align*}
The exponents in this sum are ordered by the lexicographic order of the vectors $(\varepsilon_m,\varepsilon_{m-1},\dotsc,\varepsilon_1)$, so clearly the biggest one is $\sum_{i=1}^m 2^{t_i}=n$ (when $\varepsilon_i=1$ for all $i$), the next biggest one is $\sum_{i=2}^m 2^{t_i}=n-2^{t_1}$, and so on.
\end{proof}

\begin{lem}\label{lem:S3time}
Let $n\geq 1$ and assume that $2^{d-1} \leq n < 2^d$. Then we have $t_{{\rm min}}((x+1)^n, S_3)=2^d-n$.
\end{lem}

\begin{proof}

Let $f=(x+1)^n$. Let us write the base-2 expansion of $n$ as $n=\sum_{i\in A}2^i$ for some subset $A\subset \mathbb N\cup \{0\}$, and group the elements of $A$ into intervals of consecutive integers:
$$A=\bigcup_{i=1}^m\{t_i,t_{i}+1,\dotsc ,s_{i}\},$$
where $0\leq t_1 \leq s_1 <t_2 \leq s_2 < \dotsc < t_{m} \leq s_m$  and $s_i \leq t_{i+1}-2$, for all $1 \leq i \leq m-1$.

For brevity, for any integers $a \leq b$, let us denote
$$S(a,b)=\sum_{i=a}^{b}2^i.$$
Note that $S(a,b)=2^{b+1}-2^a$. Then we have
\begin{equation}
n=\sum_{i=1}^m S(t_i,s_i). \label{eq:n_expansion}
\end{equation}
Consider the following numbers:
\begin{eqnarray*}
g_0&=&2^{t_1}, \\
g_1&=&S(s_1+1,t_2-1), \\
g_2&=&S(s_2+1,t_3-1),\\
&\vdots&\\
g_{m-1}&=&S(s_{m-1}+1,t_m-1).
\end{eqnarray*}
For $-1\leq k \leq m-1$, let us define $G_k=\sum_{i=0}^k g_i$. We will prove the following claims:
\begin{enumerate}
    \item For all $-1\leq k\leq m-1$, we have $\deg \left(S_3^{ G_k} (f)\right) = n-\sum_{i=1}^{k+1}S(t_i,s_i)$. 
    \item For all $0\leq k \leq m-1$, we have $\deg \left(S_3^{G_{k}-1} (f)\right)> \deg \left(S_3^{ G_{k}}(f)\right)$.
\end{enumerate}

The first claim will be proved by induction on $k$, and the second one will be proved as a byproduct of the inductive proof.

The case $k=-1$ is trivial: $G_{-1}=0$, so $S_3^{ G_{-1}}(f)=f$, and indeed $\deg(f) = n$.
Now, for any $-1\leq k \leq m-2$, let us assume that the first claim is true for $k$, and prove both claims for $k+1$.
Let $D=\deg \left(S_3^{ G_k} (f)\right)$. By our inductive hypothesis, $D= n-\sum_{i=1}^{k+1}S(t_i,s_i)$.
By Lemma \ref{lem:restriction}, we have 
\begin{equation} \label{eq:S3^G_k}
S_3^{ G_k}(f)=(x+1)^{G_k}f|_{\leq D}=(x+1)^{n+G_k}|_{\leq D}.
\end{equation}
Let us find the base-2 expansion of $n+G_k$:
\begin{align} 
n+G_k&=\sum_{i=1}^mS(t_i,s_i)+2^{t_1}+\sum_{i=1}^k S(s_i+1,t_{i+1}-1) \nonumber \\
&=2^{t_1}+\left (\sum_{i=1}^{k+1}S(t_i,s_i) + \sum_{i=1}^k S(s_i+1,t_{i+1}-1) \right )+\sum_{i=k+2}^mS(t_i,s_i) \nonumber \\
&=2^{t_1}+S(t_1,s_{k+1})+ \sum_{i=k+2}^m S(t_i,s_i) = 2^{s_{k+1}+1}+\sum_{i=k+2}^m S(t_i,s_i). \label{eq:n+G}
\end{align}
From this base-2 expansion and Lemma \ref{lem:expansion}, we have
\begin{equation} \label{eq:expansion_n+Gk}
(x+1)^{n+G_k}=x^{n+G_k}+x^{n+G_k-2^{s_{k+1}+1}}+x^{n+G_k-2^{t_{k+2}}}+\dotsc
\end{equation}
By (\ref{eq:n_expansion}) and (\ref{eq:n+G}), the second highest power in this expansion is
\begin{equation} \label{eq:second_power}
n+G_k-2^{s_{k+1}+1}=\sum_{i=k+2}^m S(t_i,s_i)=D.
\end{equation}
Hence, by (\ref{eq:S3^G_k}), (\ref{eq:expansion_n+Gk}) and (\ref{eq:second_power}):
$$S_3^{ G_k}(f)=x^{n+G_k-2^{s_{k+1}+1}}+x^{n+G_k-2^{t_{k+2}}}+\dotsc$$
The gap between the two leading powers in $S_3^{ G_k}(f)$ is therefore
$$2^{t_{k+2}}-2^{s_{k+1}+1}=S(s_{k+1}+1,t_{k+2}-1)= g_{k+1}.$$
We conclude:
$$ S_3^{ G_k}(f) = X^D+x^{D-g_{k+1}}+\dotsc $$
Letting $h=S_3^{ G_k}(f) +X^D$, we have $\deg(h) = D-g_{k+1}$. By Lemma \ref{lem:iterations_formula}, for all $1\leq i \leq g_{k+1}$ we have 
$$ S_3^{G_k+i}(f) = X^D + (x+1)^i h .$$
In particular, 
$$\deg \left(S_3^{G_{k+1}-1}(f)\right)= \deg \left(S_3^{G_{k}+g_{k+1}-1}(f)\right)=D>\deg \left(S_3^{G_{k}+g_{k+1}}(f)\right)=\deg \left(S_3^{ G_{k+1}}(f)\right).$$
This proves the second claim.
Finally, by (\ref{eq:S3^G_k}) and Lemma \ref{lem:restriction},
\begin{align}
S_3^{ G_{k+1}}(f)&= (S_3^{ g_{k+1}}\circ S_3^{ G_k})(f)  \nonumber \\
&= \left ((x+1)^{g_{k+1}} ( (x+1)^{n+G_k} |_{\leq D})\right ) |_{\leq D} \nonumber \\
&= ((x+1)^{n+G_k+g_{k+1}})|_{\leq D} = ((x+1)^{n+G_{k+1}})|_{\leq D}.  \label{eq:newdeg}
\end{align}
As equations (\ref{eq:expansion_n+Gk}) and (\ref{eq:second_power}) were proved independently of the inductive hypothesis, they hold for $k+1$ as well, hence
$(x+1)^{n+G_{k+1}}=x^{n+G_{k+1}}+x^{D'}+\dotsc$ where $D'=n-\sum_{i=1}^{k+2}S(t_i,s_i)$.
Since $D'<D$, we get by (\ref{eq:newdeg}) that $\deg \left(S_3^{ G_{k+1}}(f)\right)=D'$. Our induction is therefore complete.

Now, setting $k=m-1$, we get
$$\deg \left(S_3^{G_{m-1}-1}(f)\right)>\deg \left(S_3^{ G_{m-1}}(f)\right)=n-\sum_{i=1}^{m}S(t_i,s_i)=0.$$
Hence $t_{\rm min}(f,S_3)=G_{m-1}$. 
By (\ref{eq:n+G}), noting that $s_m=d-1$, we have $$n+G_{m-1}=2^{s_m+1}=2^d.$$
Hence $G_{m-1}=2^d-n$, and the proof of the lemma is complete.
\end{proof}

\begin{coro} \label{coro:stoppingtime}
Let $(a,b) \in (\mathbb{N}\cup\{0\}) \times (\mathbb{N}\cup\{0\}) \setminus \{(0,0)\}$ and $n \geq 1$. Let $d$ be an integer such that $2^d<n(a+b)\leq 2^{d+1}$. Then 
$$t_{\rm min}(f_{a,b,n})=2^{d+2}+(a-b)n.$$
\end{coro}
\begin{proof}
Since $n,a,b,d$ are integers, we have $2^d\leq n(a+b)-1<2^{d+1}$. By Lemma \ref{lem:S3time}, we have $$t_{\rm min}((x+1)^{na+nb-1},S_3)=2^{d+1}-n(a+b)+1.$$ Hence, by Lemma \ref{lem:fab_reduce}, we get 
\begin{align*}
    t_{\rm min}(f_{a,b,n})&= 2t_{\rm min}((x+1)^{na+nb-1},S_3)+3na+nb-2 \\
    &= 2(2^{d+1}-n(a+b)+1)+3na+nb-2 \\
    &=2^{d+2}+na-nb. \qedhere
\end{align*}
\end{proof}

Finally, we prove Theorem \ref{main2} of the introduction concerning the existence of arithmetic sequences of unbounded length in $(t_{{\rm min}}(f_{a,b,n}))_{n \geq 1}$:

\begin{theo}
Let $(a,b) \in (\mathbb{N}\cup\{0\}) \times (\mathbb{N}\cup\{0\}) \setminus \{(0,0)\}$. The sequence $(t_{{\rm min}}(f_{a,b,n}))_{n \geq 1}$ has arithmetic sequences of unbounded length and with common difference $a-b$.
\end{theo}
\begin{proof}
For any $d \geq \log_{2}(a+b)$ and  $  \left \lfloor \frac{2^d}{a+b}   \right \rfloor +1 \leq n \leq \left \lfloor \frac{2^{d+1}}{a+b}  \right \rfloor$, we have  $2^d<n(a+b)\leq 2^{d+1}$, so by Corollary \ref{coro:stoppingtime}, $t_{\rm min}(f_{a,b,n})=2^{d+2}+(a-b)n$. Hence,
$$\left(t_{{\rm min}}(f_{a,b,n})\right)_{n=\left\lfloor \frac{2^d}{a+b}\right\rfloor +1}^{\left\lfloor \frac{2^{d+1}}{a+b}\right\rfloor}= \left(2^{d+2}+(a-b)n\right)_{n=\left\lfloor \frac{2^d}{a+b}\right\rfloor +1}^{\left\lfloor \frac{2^{d+1}}{a+b}\right\rfloor},
$$ which is an arithmetic sequence with common difference $a-b$ and of length $\left\lfloor \frac{2^{d+1}}{a+b}\right\rfloor- \left\lfloor\frac{2^{d}}{a+b} \right\rfloor$. It remains to note that $\left\lfloor \frac{2^{d+1}}{a+b}\right\rfloor- \left\lfloor\frac{2^{d}}{a+b} \right\rfloor \rightarrow \infty$ as $d \rightarrow \infty$.
\end{proof}

\section{Open Questions}\label{sec:open}

We conclude this work by pointing out some natural open questions for further study.
For any $d \in \mathbb{N}$, let $$\sigma(d) = \max\{t_{{\rm min}}(f) \mid f \in \F_2[x],\deg(f) = d\}$$ be the maximal stopping time of polynomials of degree $d$. Theorem \ref{conj_main} provides a bound of $O(d^{1.5})$ on the asymptotic growth of $\sigma(d)$. We do not know whether this bound is optimal.

\begin{question}\label{worse}
What is the asymptotic growth of $\sigma(d)$? \end{question}

One can also consider the {\bf average} stopping time of polynomials of degree $d$:
$$\rho(d) = \frac{1}{2^{d}}\sum_{f \in \F_2[x], \deg(f) = d}t_{{\rm min}}(f).$$
As mentioned in the introduction, for the arithmetic Collatz map ${\cal{C}}$, it is conjectured that the average stopping time of integers of $d$ digits grows linearly in $d$. Experimental data suggests that the same holds for the Collatz map on $\F_2[x]$, but we do not have a proof of this. We thus formulate the following conjecture:

\begin{conj}\label{average}
The average stopping time $\rho(d)$ of the Collatz map $T$ on $\F_2[x]$ grows linearly in $d$.
\end{conj}

In contrast, regarding Question \ref{worse}, we suspect that the growth of $\sigma(d)$ is {\bf not} linear in $d$: Computer experiments run by Alon Rosenperl suggest that $\sigma(d) \geq d\cdot \log(d)$ for any large enough $d$. We do not have a proof of these findings. 

In the recent work \cite{BP23}, the main result of Hicks et al. was generalized from $\F_2[x]$ to $\F_p[x]$, for any prime $p$, by considering the following polynomial Collatz map $T$ on $\F_p[x]$:

$$
T(f)=\begin{cases}
f \cdot (x+1)-f_0 &\textrm{if}\,\, f_0 \neq 0,\\
\frac{f}{x} & \textrm{otherwise}.
\end{cases}
$$
(Clearly, for $p = 2$, this map coincides with the map $T$ of Hicks et al.) A polynomial $f \in \F_p[x]$ is called {\bf periodic} (with respect to $T$) if there exists $n \geq 1$ such that $T^n(f) = f$. A polynomial $f \in \F_p[x]$ is called {\bf pre-periodic} (with respect to $T$) if there exists $k \geq 1$ such that $T^k(f)$ is periodic. 

Let $p > 2$ be a prime. It is observed in \cite{BP23} that the $T$-iterations of non-zero polynomials in $\F_p[x]$ do not always contain $1$, however, every polynomial $f \in \F_p[x]$ is pre-periodic with respect to $T$. Let us define the stopping time of a polynomial $f \in \F_p[x]$ as the minimal $k$ for which $T^k(f)$ is periodic. The main result of \cite{BP23} is a quadratic bound on the stopping time: If $f \in \F_p[x]$ is a non-zero polynomial of degree $d$, then the stopping time of $f$ is at most $p(d^2+d)-d$ \cite[Theorem 1.1]{BP23}. We do not know whether this bound can be improved, for any $p > 2$. We thus pose the following question, which generalizes Question \ref{worse}: 

\begin{question}Let $p$ be a prime. What is the asymptotic growth in $d$ of the maximal stopping time of polynomials in $\F_p[x]$ of degree $d$? \end{question}

Similarly, one may ask:

\begin{question}Let $p$ be a prime. What is the asymptotic growth in $d$ of the average stopping time of polynomials in $\F_p[x]$ of degree $d$? \end{question}

\vspace{20pt}

\end{document}